\documentclass[a4paper, 12pt]{article}

\newif\ifAMS
\IfFileExists{amssymb.sty}
  {\AMStrue\usepackage{amssymb}}{}
\usepackage{latexsym}

\usepackage[exercise]{amsthm}
\usepackage{amsmath}

\usepackage{graphicx,textcomp,booktabs,amsmath,pst-node}
\usepackage{mathptmx,courier}
\usepackage[scaled]{helvet}

\newtheorem{thm}{Theorem}[section]
\newtheorem{cor}[thm]{Corollary}
\newtheorem{prop}{Proposition}[section]
\newtheorem{lem}[thm]{Lemma}
\newtheorem{rem}{Remark}[section]

\theoremstyle{definition}
\newtheorem{defn}{Definition}[section]

\numberwithin{equation}{section}

\begin{document}

\tolerance=10000

\bigskip
\bigskip
\bigskip
\bigskip
\bigskip

\title{Bridgeland Stability Conditions \\ 
on (twisted) Kummer surfaces}

\bigskip
\bigskip
\bigskip
\bigskip

\author{Magnus Engenhorst\footnote{magnus.engenhorst@math.uni-freiburg.de}\\
Mathematical Institute, University of Freiburg\\
Eckerstrasse 1, 79104 Freiburg, Germany\\[5mm]}

\maketitle

\begin{abstract}
We construct a topological embedding of the maximal connected component of Bridgeland stability conditions of a (twisted) Abelian surface into the distinguished connected component of the stability manifold of the associated (twisted) Kummer surface. We use methods developed for orbifold conformal field theories.  
\end{abstract}

\bigskip

\setcounter{equation}{0}

\section{Introduction}

Mirror symmetry and Bridgeland's stability conditions are mathematical theories motivated by superconformal field theories (SCFT) associated to Calabi-Yau varieties. It is a non-trivial task to give the geometric interpretation of a SCFT a rigorous meaning. This is understood in the case of complex tori ~\cite{5,71} and progress has been made for certain K3 surfaces \cite{6} using realizations by non-linear $\sigma$ models. The case of Calabi-Yau threefolds turned out to be much harder: In fact, up to now there is no example of a stability condition. At least there is a concrete conjecture ~\cite{3}. There are also results for SCFTs on Borcea-Voisin threefolds \cite{18,9}. It would be interesting to study the question of quantum corrections of the central charge in an example. \\

We are interested in the case of (projective) Kummer surfaces. The aim is to lift results of ~\cite{6,7} for the associated SCFT to the space of stability conditions. In section 6 we construct a topological embedding of the unique maximal connected component $Stab^{\dagger}(A)$ of Bridgeland stability conditions of an Abelian surface A into the distinguished connected component $Stab^{\dagger}(X)$ of the stability manifold of the associated projective Kummer surface X (Theorem 5.7). We show that the group of deck transformations of $Stab^{\dagger}(A)$ (generated by the double shift) is isomorphic to a subgroup of the group of deck transformations of $Stab^{\dagger}(X)$ (Proposition 5.7).\\

This work is based on results for the embedding of the moduli space of SCFTs on complex tori into the moduli space of SCFTs on the associated Kummer surfaces given in ~\cite{6,7}. Crucial for this paper is the observation confirmed in these works that there are no ill-defined SCFTs coming from the complex torus. Rephrased in mathematical terms this is corollary 3.4. The important role of ill-defined SCFTs was rederived by Bridgeland in \cite{90}. For this issue see also \cite{110}. The mentioned embedding also holds true for twisted surfaces that include in their geometrical data a rational B-field $B\in H^{2}(X,\mathbb{Q})$ (or Brauer class). Daniel Huybrechts used generalized Calabi-Yau structures \cite{10} to describe moduli spaces of N=(2,2) SCFTs as moduli spaces of generalized Calabi-Yau structures in \cite{20}. In the case of a Kummer surface we have a canonical B-field for the orbifold conformal field theory ~\cite{6} that is compatible with the generalized Calabi-Yau structures. \\             

The paper is organised as follows:\\ 

In section 2 we review facts about the moduli space of superconformal field theories for complex tori and K3 surfacs. In Section 3 we discuss the results of \cite{6,7} for orbifold conformal field theories on Kummer surfaces. As explained above we use these results in section 5.  Generalized Calabi-Yau structures serve as geometric counterpart of SCFTs with B-fields. We introduce this notion in section 4. The results of this paper can be found in section 5. There we review the result of Bridgeland that a component of the space of stability conditions for algebraic K3 surfaces is a covering space of a subspace of the complexified even cohomology lattice. We use this result and the results from section 3 to study stability conditions in the distinguished connected component of the stability manifold of projective (twisted) Kummer surfaces induced from the Abelian surfaces.

\section{Moduli spaces of superconformal field theories}

In this section we discuss the moduli space of N=(4,4) SCFTs with central charge $c=6$. We follow in this section the version of \cite{6,7}. For a pedagogical introduction see \cite{666}. Let X be a two-dimensional Calabi-Yau manifold, i.e. a complex tori or a K3 surface. We have a pairing induced by the intersection product on the even cohomology $H^{even}(X,\mathbb{R})\cong \mathbb{R}^{4,4+\delta}$. We choose a marking, that is an isometry $H^{even}(X,\mathbb{Z})\cong L$ where L is the unique even unimodular lattice $\mathbb{Z}^{4,4+\delta}$ with $\delta=0$ for a complex torus and $\delta=16$ for a K3 surface. In the latter case this is of course just the K3 lattice $4U\oplus 2(-E_{8})$. The moduli space of SCFTs associated to complex tori or K3 surfaces are given by the following

\begin{thm} \cite{92}
Every connected component of the moduli space of SCFTs associated to Calabi-Yau 2-folds is either of the form $\mathcal{M}_{tori}=\mathcal{M}^{0}$ or $\mathcal{M}_{K3}=\mathcal{M}^{16}$ where:
\begin{eqnarray}
\mathcal{M}^{\delta}\cong O^{+}(4,4+\delta;\mathbb{Z})\backslash O^{+}(4,4+\delta;\mathbb{R})/SO(4)\times O(4+\delta). \nonumber
\end{eqnarray}
\end{thm}    

Points $x\in \tilde{\mathcal{M}}^{\delta}$ in the Grassmannian 
\begin{eqnarray}
\tilde{\mathcal{M}}^{\delta}=O^{+}(4,4+\delta;\mathbb{R})/SO(4)\times O(4+\delta) \nonumber
\end{eqnarray}
correspond to positive definite oriented four-planes in $\mathbb{R}^{4,4+\delta}$ whose position is given by its relative position to the reference lattice L.\\

Let us choose a marking $H^{2}(X,\mathbb{Z})\cong \mathbb{Z}^{3,3+\delta}$. The Torelli theorem \cite{18001, 28} then tells us that complex structures on two-dimensional complex tori or K3 surfaces X are in 1:1 correspondence with positive definite oriented two-planes $\Omega\subset H^{2}(X,\mathbb{Z})\otimes{R}\cong \mathbb{R}^{3,3+\delta}$ that are specified by its relative position to $\mathbb{Z}^{3,3+\delta}$. 

\begin{defn}
Let $x\subset H^{even}(X,\mathbb{Z})\otimes \mathbb{R}$ be a positive oriented four-plane specifying a SCFT on X. A \textit{geometric interpretation of this SCFT} is a choice of null vectors $\upsilon^{0}, \upsilon\in H^{even}(X,\mathbb{Z})$ along with a decomposition of x into two perpendicular oriented two-planes $x=\Omega\bot \mho$ such that $\left\langle \upsilon^{0},\upsilon^{0}\right\rangle=\left\langle \upsilon,\upsilon\right\rangle=0$, $\left\langle \upsilon^{0},\upsilon\right\rangle=1$, and $\Omega\bot \upsilon^{0},\upsilon$. 
\end{defn}

\begin{lem}\cite{92}
\label{am}
Let $x\subset H^{even}(X,\mathbb{Z})\otimes \mathbb{R}$ be a positive definite oriented four-plane with geometric interpretation $\upsilon^{0}, \upsilon\in H^{even}(X,\mathbb{Z})$, where $\upsilon^{0}, \upsilon$ are interpreted as generators of $H^{0}(X,\mathbb{Z})$ and $H^{4}(X,\mathbb{Z})$, respectively, and a decomposition $x=\Omega\bot \mho$. Then one finds an unique $\omega\in H^{even}(X,\mathbb{Z})\otimes \mathbb{R}$ and $B\in H^{even}(X,\mathbb{Z})\otimes \mathbb{R}$ with
\begin{eqnarray}
\mho=\mathbb{R}\left\langle \omega-\left\langle B, \omega\right\rangle\upsilon, \xi_{4}=\upsilon^{0}+B+\left(V-\frac{1}{2}\left\langle B,B\right\rangle\right)\upsilon\right\rangle
\end{eqnarray}
\end{lem} 
with $\omega, B\in H^{2}(X,\mathbb{R}):=H^{even}(X,\mathbb{R})\cap \upsilon^{\perp}\cap(\upsilon^{0})^{\perp}$ ,$V\in\mathbb{R}_{+}$ and $\omega^{2}\in\mathbb{R}_{+}$. B and V are determined uniquely and $\omega$ is unique up to scaling.\\ 

The picture is that a SCFT associated to a Calabi-Yau 2-fold can be realized by a non-linear $\sigma$ model. It is important to note that the mentioned moduli space of SCFTs associated to K3 surfaces also contains ill-defined conformal field theories. Namely, a positive definite oriented four-plane $x\in \tilde{\mathcal{M}}^{16}$ corresponds to such a theory if and only if there is a class $\delta\in H^{even}(X,\mathbb{Z})$ with $\delta\bot x$ and $\left\langle \delta,\delta\right\rangle=-2$. String theory tells us that the field theory gets extra massless particles at these points in the moduli space and breaks down. For physical details see \cite{923}. For complex tori there are no such ill-defined SCFTs.

\section{Orbifold conformal field theories on K3}
\label{chapterdrei}

We are interested in SCFTs with geometric interpretations on Kummer surfaces coming from orbifolding of SCFTs on complex tori since later we want to induce stability conditions on projective Kummer surfaces from the associated Abelian surfaces. \\

We consider a complex torus T with the standard $G=\mathbb{Z}_{2}$ action and its associated Kummer surface X. We have a minimal resolution of the sixteen singularities:
\begin{eqnarray}
X:=\widetilde{T/G}\longrightarrow T/G. \nonumber
\end{eqnarray}
This resolution introduces 16 rational two-cycles which we label by $\mathbb{F}_{2}^{4}$ and we denote their Poincar\'e duals by $E_{i}$ with $i\in \mathbb{F}_{2}^{4}$. The Kummer lattice $\Pi$ is the smallest primitive sublattice of the Picard lattice Pic(X)=NS(X) containing $\left\{E_{i}|i\in \mathbb{F}_{2}^{4}\right\}$. It is spanned by $\left\{E_{i}|i\in \mathbb{F}_{2}^{4}\right\}$ and $\left\{1/2\sum_{i\in H}E_{i}|H\subset \mathbb{F}^{4}_{2} \text{ a hyperplane}\right\}$ \cite{60}. (For a review see e.g. \cite{930}). We want to find an injective map from the moduli space of SCFTs on a two-dimensional complex torus T to the moduli space of SCFTs on the corresponding Kummer surface X. This was done by Nahm and Wendland \cite{6,7} generalizing results of Nikulin \cite{60}: \\           

Let $\pi:T\rightarrow X$ be the induced rational map of degree 2 defined outside the fixed points of the $\mathbb{Z}_{2}$ action. The induced map on the cohomology gives an embedding $\pi_{*}:H^{2}(T,\mathbb{Z})(2)\hookrightarrow H^{2}(X,\mathbb{Z})$ \cite{60, 50}.\footnote{Here and in the following L(2) means a lattice L with quadratic form scaled by 2.} We define $K:=\pi_{*}H^{2}(T,\mathbb{Z})$. The lattice K obeys $K\oplus \Pi\subset H^{2}(X,\mathbb{Z})\subset K^{*}\oplus \Pi^{*}$ where $K\oplus \Pi\subset H^{2}(X,\mathbb{Z})$ is a primitive sublattice with the same rank as $H^{2}(X,\mathbb{Z})$. $H^{2}(X,\mathbb{Z})$ is even and unimodular. Let $\mu_{1},\ldots,\mu_{4}$ denote generators of $H^{1}(T,\mathbb{Z})$. This embedding defines the isomorphism

\begin{eqnarray}
\label{niku}
\gamma:K^{*}/K&\longrightarrow &\Pi^{*}/\Pi \\
\frac{1}{2}\pi_{*}(\mu_{j}\wedge\mu_{k}) & \longmapsto & \frac{1}{2}\sum_{i\in P_{jk}}E_{i} \nonumber
\end{eqnarray}                 
where $P_{jk}=\left\lbrace a=(a_{1},a_{2},a_{3},a_{4})\in\mathbb{F}_{2}^{4}|a_{l}=0, \forall l\neq j,k\right\rbrace$ with $j,k\in\left\lbrace 1,2,3,4\right\rbrace $. Conversely, with this isomorphism we can describe the lattice $H^{2}(X,\mathbb{Z})$ using  

\begin{thm}\cite{940,941}
\label{nikulin}
Let $\Lambda$ $\subset$ $\Gamma$ be a primitive, non-degenerate sublattice of an even, unimodular lattice $\Gamma$ and its dual $\Lambda^{*}$, with $\Lambda\hookrightarrow\Lambda^{*}$ given by the  form on $\Lambda$. Then the embedding $\Lambda\hookrightarrow\Gamma$ with $\Lambda^{\bot}\cap\Gamma\cong V$ is specified by an isomorphism $\gamma:\Lambda^{*}/\Lambda\rightarrow V^{*}/V$ such that the induced quadratic forms obey $q_{\Lambda}=-q_{V}\circ\gamma$. Moreover,
\begin{eqnarray}
\Gamma\cong\left\{(\lambda, v)\in \Lambda^{*}\oplus V^{*}|\gamma(\bar{\lambda})=\bar{v}\right\}.
\end{eqnarray} 
\end{thm}     
Here $\bar{l}$ is the projection of $l\in L^{*}$ onto $L^{*}/L$. We find in our case
\begin{eqnarray}
H^{2}(X,\mathbb{Z})\cong\left\{(\kappa,\pi)\in K^{*}\oplus \Pi^{*}|\gamma(\bar{\kappa})=\bar{\pi}\right\}. \nonumber
\end{eqnarray}
Hence $H^{2}(X,\mathbb{Z})$ is generated by
\begin{enumerate}
\item $\pi_{*}H^{2}(T,\mathbb{Z})\cong H^{2}(T,\mathbb{Z})(2)$,
\item the elements of the Kummer lattice $\Pi$,
\item and forms of the form $\frac{1}{2}\pi_{*}(\mu_{j}\wedge\mu_{k})+\frac{1}{2}\sum_{i\in P_{jk}}E_i$.
\end{enumerate}
    
Let $\upsilon^{0}$ respectively $\upsilon$ be generators of $H^{0}(T,\mathbb{Z})$ respectively $H^{4}(T,\mathbb{Z})$. The next step to find the geometric interpretation of the orbifold conformal field theory of a SCFT on a two-dimensional complex torus is to note that $\pi_{*}\upsilon,\pi_{*}\upsilon^{0}\in H^{even}(X,\mathbb{Z})$ generate a primitive sublattice with quadratic form
\begin{eqnarray}
\begin{pmatrix}
0 & 2\\
2 & 0
\end{pmatrix}. \nonumber
\end{eqnarray}
The minimal primitive sublattice $\hat{K}$ containing $\pi_{*}H^{even}(T,\mathbb{Z})\subset H^{even}(X,\mathbb{Z})$ thus obeys
\begin{eqnarray}
\hat{K}^{*}/\hat{K}\cong K^{*}/K\times \mathbb{Z}^{2}_{2}\cong \Pi^{*}/\Pi\times \mathbb{Z}^{2}_{2}. \nonumber
\end{eqnarray}  
By theorem $\ref{nikulin}$ this means that $\hat{K}$ and $\Pi$ cannot be embedded in $H^{even}(X,\mathbb{Z})$ as orthogonal sublattices. Hence $H^{0}(X,\mathbb{Z})\oplus H^{4}(X,\mathbb{Z})$ cannot be a sublattice of $\hat{K}$. We choose as generators of $H^{0}(X,\mathbb{Z})$ and $H^{4}(X,\mathbb{Z})$: 
\begin{eqnarray}
\hat{\upsilon}&:=&\pi_{*}\upsilon, \\ \nonumber
\hat{\upsilon}^{0}&:=&\frac{1}{2}\pi_{*}\upsilon^{0}-\frac{1}{4}\sum_{i\in \mathbb{F}^{4}_{2}}E_{i}+\pi_{*} \upsilon \nonumber
\end{eqnarray}
We define $\hat{E}_{i}:=-\frac{1}{2}\hat{\upsilon}+E_{i}$, where $E_{i}\perp \hat{K}$. 

\begin{lem}\cite{6,7}
\label{lattice}
The lattice generated by $\hat{\upsilon}$, $\hat{\upsilon}^{0}$ and 
\begin{eqnarray}
\label{latticelattice}
\left\{\frac{1}{2}\pi_{*}(\mu_{j}\wedge \mu_{k})+\frac{1}{2}\sum_{i\in P_{jk}}\hat{E}_{i+l};l\in\mathbb{F}^{4}_{2}\right\}\text{and} \left\{\hat{E}_{i},i\in\mathbb{F}^{4}_{2}\right\} 
\end{eqnarray}
is isomorphic to $\mathbb{Z}^{4,20}$.
\end{lem} 
In \cite{6,7,8} it is argued that this is the unique embedding which is compatible with all symmetries of the respective SCFTs. Using the generators given in Lemma \ref{lattice} we can regard a positive definite, oriented four-plane $x\subset H^{even}(T,\mathbb{Z})\otimes \mathbb{R}$ as a four-plane in $H^{even}(X,\mathbb{Z})\otimes \mathbb{R}$. 
\begin{thm}\cite{6,7}
\label{nw}
For a geometric interpretation of a SCFT $x_{T}=\Omega\bot\mho$ on a complex torus T with $\omega, V_{T}, B_{T}$ as in Lemma $\ref{am}$ the corresponding orbifold conformal field theory $x=\pi_{*}\Omega\bot \pi_{*}\mho$ has a geometric interpretation $\hat{\upsilon}$, $\hat{\upsilon}^{0}$ with $\pi_{*}\omega, V=\frac{V_{T}}{2}, B$ where
\begin{eqnarray}
B&=&\frac{1}{2}\pi_{*}B_{T}+\frac{1}{2}B_{\mathbb{Z}}, \\ \nonumber
B_{\mathbb{Z}}&=&\frac{1}{2}\sum_{i\in \mathbb{F}^{4}_{2}}\hat{E}_{i}. \nonumber
\end{eqnarray}    
\end{thm} 
\begin{proof}
Using the embedding $H^{even}(T,\mathbb{Z})\otimes\mathbb{R}\hookrightarrow H^{even}(X,\mathbb{Z})\otimes\mathbb{R}$ given in Lemma \ref{lattice} we calculate
\begin{eqnarray}
\pi_{*}\left(\omega-\left\langle B_{T}, \omega\right\rangle\upsilon\right)&=&\pi_{*}\omega-\left\langle \pi_{*}B,\omega\right\rangle\hat{\upsilon},  \nonumber \\
\frac{1}{2}\pi_{*}\left(\upsilon^{0}+B_{T}+\left(V_{T}-\frac{1}{2}\left\|B_{T}\right\|^{2}\right)\upsilon\right)&=& 
\hat{\upsilon}^{0}+\frac{1}{2}\pi_{*}B_{T}+\frac{1}{2}B_{\mathbb{Z}} \nonumber \\ 
&+&  \left(\frac{V_{T}}{2}-\frac{1}{2}\left\|\frac{1}{2}\pi_{*}B_{T}+\frac{1}{2}B_{\mathbb{Z}}\right\|^{2}\right)\hat{\upsilon}. \nonumber
\end{eqnarray}
This proves the theorem.
\end{proof}    

For chapter 5 the following observation is crucial:
\begin{cor}\cite{6,7}
\label{korollar}
Let $x=\pi_{*}\Omega\bot \pi_{*}\mho\subset H^{even}(X,\mathbb{Z})\otimes \mathbb{R}$ be the four-plane induced from a positive-definite, oriented four-plane $x_{T}=\Omega\bot \mho\subset H^{even}(T,\mathbb{Z})\otimes \mathbb{R}$ as in Theorem \ref{nw}. 
Then $x^{\bot}\cap H^{even}(X,\mathbb{Z})$ does not contain (-2) classes.
\end{cor}
\begin{proof}
Let $\Omega$ be the positive-definite, oriented two-plane defined by the complex structure for the torus T. We choose a basis of the orthogonal complement $x^{\bot} \subset H^{even}(X,\mathbb{Z})\otimes \mathbb{R}$. For example:
\begin{enumerate}
\item $\hat{E}_{i}+\frac{1}{2}\hat{\upsilon}, i\in\mathbb{F}_{2}^{4}$,  
\item $\pi_{*}\eta_{i}-\left\langle\pi_{*}\eta_{i},B\right\rangle\hat{\upsilon}, i=1,\ldots,3$,
\item $\hat{\upsilon}^{0}+B-\left(V+\frac{1}{2}\left\|B\right\|^{2}\right)\hat{\upsilon}$.   
\end{enumerate}
The $\eta_{i}, i=1,\ldots,3$ are an orthogonal basis of the orthogonal complement of $\text{span}_{\mathbb{R}}\langle\omega,\Omega\rangle$ in $H^{2}(T,\mathbb{Z})\otimes \mathbb{R}$. Then the $\pi_{*}\eta_{i}, i=1,\ldots,3$ build together with the sixteen $E_{i}, i\in\mathbb{F}_{2}^{4}$ an orthogonal basis of the orthogonal complement of $\text{span}_{\mathbb{R}}\langle\pi_{*}\omega,\pi_{*}\Omega\rangle$ in $H^{2}(X,\mathbb{Z})\otimes \mathbb{R}$ with $\omega$ as in Lemma \ref{am}. B is as in Theorem $\ref{nw}$. Note that $\left\langle E_{i},E_{i}\right\rangle=-2$ but $E_{i}$ is not an element of our lattice. If we then try to build a (-2) class in $x^{\bot}$ from our ansatz we run into contradictions.  
\end{proof}

\section{Generalized Calabi-Yau Structures}

In this section we introduce generalized Calabi-Yau structures of Hitchin ~\cite{10} following ~\cite{20,160}. This is also relevant for stability conditions on twisted surfaces as we will see in section 6. \\

The Mukai pairing on the even integral cohomology $H^{even}(X,\mathbb{Z})=H^{0}(X,\mathbb{Z})\oplus H^{2}(X,\mathbb{Z})\oplus H^{4}(X,\mathbb{Z})$ is defined by
\begin{eqnarray}
	\left\langle (a_{0},a_{2},a_{4}), (b_{0},b_{2},b_{4})\right\rangle:=-a_{0}\wedge b_{4}+a_{2}\wedge b_{2}-a_{4}\wedge b_{0}.  \nonumber
\end{eqnarray}    

For an Abelian or K3 surface X the Mukai lattice is $H^{even}(X,\mathbb{Z})$ equipped with the Mukai pairing that differs from the intersection pairing in signs. Note that the hyperbolic lattice $U$ with basis $\upsilon,\upsilon^{0}$ is isomorphic to $-U$ via
\begin{eqnarray}
\upsilon&\longmapsto& -\upsilon, \nonumber \\
\upsilon^{0}&\longmapsto&\upsilon^{0}.  \nonumber
\end{eqnarray}
From now on we will work in the Mukai lattice.\\  

\begin{defn}
Let $\Omega$ be a holomorphic 2-two form on an Abelian or K3 surface X defining a complex structure. For a rational B-field $B\in H^{2}(X,\mathbb{Q})$ a \textit{generalized Calabi-Yau structure on X} is given by 
\begin{eqnarray}
\varphi:=exp(B)\Omega=\Omega+B\wedge\Omega\in H^{2}(X)\oplus H^{4}(X). \nonumber
\end{eqnarray}
\end{defn}

We define a Hodge structure of weight two on the Mukai lattice by
\begin{eqnarray}
	\widetilde{H}^{2,0}(X):=\mathbb{C}\left[\varphi\right] \nonumber
\end{eqnarray}
We write $\widetilde{H}(X,B, \mathbb{Z})$ for the lattice equipped with this Hodge structure and the Mukai pairing. 

\begin{defn}
Let $\varphi=exp(B)\Omega$ be a generalized Calabi-Yau structure. The \textit{generalized transcendental lattice $T(X,B)$} is the minimal primitive sublattice of $H^{2}(X,\mathbb{Z})\oplus H^{4}(X,\mathbb{Z})$, such that $\varphi\in T(X,B)\otimes\mathbb{C}$. 
\end{defn}
$T(X,0)=T(X)=NS(X)^{\perp}$ is the transcendental lattice and $NS(X)=H^{1,1}(X)\cap H^{2}(X,\mathbb{Z})$ is the N\'eron-Severi lattice.

\begin{defn}
Let X be a smooth complex projective variety. The \textit{(cohomological) Brauer group} is the torsion part of $H^{2}(X,\mathcal{O}^{*}_{X})$ in the analytic topology: $Br(X)=H^{2}(X,\mathcal{O}^{*}_{X})_{tor}$.\footnote{Equivalently, we could define the Brauer group as the torsion part of $H_{et}^{2}(X,\mathcal{O}_{X}^{*})$ in the \'Etale topology.}
\end{defn}   
 For an introduction to Brauer classes see ~\cite{30} or ~\cite{33}. Eventually we introduce twisted surfaces:
\begin{defn}
A \textit{twisted Abelian or K3 surface (X,$\alpha$)} consists of an Abelian or K3 surface X together with a class $\alpha \in Br(X)$. Two twisted surfaces $(X,\alpha), (Y,\alpha')$ are isomorphic if there is an isomorphism $f:X\cong Y$ with $f^{*}\alpha'=\alpha$.
\end{defn}

The exponential sequence
\begin{eqnarray}
0\longrightarrow \mathbb{Z}\longrightarrow\mathcal{O}_{X}\longrightarrow\mathcal{O}_{X}^{*}\longrightarrow 1 \nonumber
\end{eqnarray}
gives the long exact sequence
\begin{eqnarray}
\longrightarrow H^{2}(X,\mathbb{Z})\longrightarrow H^{2}(X,\mathcal{O}_{X})\longrightarrow H^{2}(X,\mathcal{O}_{X}^{*})\longrightarrow H^{3}(X,\mathbb{Z})\longrightarrow. \nonumber
\end{eqnarray}
For an Abelian or K3 surface $H_{1}(X,\mathbb{Z})$ and therefore $H^{3}(X,\mathbb{Z})$ is torsion free. So an n-torsion element of $H^{2}(X,\mathcal{O}_{X}^{*})$ is always in the image of the exponential map for a $B^{0,2}\in H^{2}(X,\mathcal{O}_{X})$ such that $nB^{0,2}\in H^{2}(X,\mathbb{Z})$ for a positive integer n. For a rational B-field $B\in H^{2}(X,\mathbb{Q})$ we use the induced homomorphism
\begin{eqnarray}
	B:T(X)&\longrightarrow& \mathbb{Q} \nonumber \\
	\gamma&\longmapsto& \int_{X}\gamma\wedge B \nonumber
\end{eqnarray}
(modulo $\mathbb{Z}$) to introduce
\begin{eqnarray}
\label{talpha}
	T(X,\alpha_{B})&:=&ker\left\{B:T(X)\rightarrow \mathbb{Q}/\mathbb{Z}\right\}.
\end{eqnarray}
The details can be found in ~\cite{20,40}. \\

\section{Stability conditions on Kummer surfaces}
\label{chapterfive}

We have an embedding of the moduli space of orbifold conformal field theories corresponding to SCFTs associated to Kummer surfaces in the moduli space of SCFTs on K3 surfaces. We are interested in the question if this embedding has a lift to Bridgeland stability conditions. In the following we show that this is indeed the case. \\

The abstract lattice $\mathbb{Z}^{4,20}$ is isometric to the even cohomology lattice $H^{even}(X,\mathbb{Z})$ equipped with the Mukai (or intersection) pairing such that the generators $\upsilon^{0}$ respectively $\upsilon$ of the hyperbolic lattice $U$  are identified with $1\in H^{0}(X,\mathbb{Z})$ respectively $\left[pt\right]\in H^{4}(X,\mathbb{Z})$ (using Poincar\'e duality). The lattice $\mathbb{Z}^{4,20}$ is also isometric to the lattice defined in Lemma $\ref{lattice}$. We will switch in this section between these isometries. \\     

Moduli spaces of N=(2,2) SCFTs can be seen as moduli spaces of generalized Calabi-Yau structures \cite{20}. Since we have an embedding of orbifold conformal field theories it is natural to ask if there is a relation between the structures we introduced in section 4 for an Abelian surface A and the associated Kummer surface $X=Km\text{ A}$.      

\begin{lem}
\label{meinlemma}
Let $(A, \alpha_{B_{A}})$ be a twisted Abelian surface and $(X,\alpha_{B})$ the associated twisted Kummer surface with B-field lift $B_{A}\in H^{2}(A,\mathbb{Q})$ as described above and B as in Theorem $\ref{nw}$. Then we have a Hodge isometry $T(A,B_{A})(2)\cong T(X,B)$.
\end{lem}
\begin{proof}
For a rational B-field B we have a Hodge isometry 
\begin{eqnarray}
	T(X, \alpha_{B})\cong T(X, B) \nonumber
\end{eqnarray}
This was proven for K3 surfaces in ~\cite{20} and also works for Abelian surfaces. The isomorphism in Theorem $\ref{nikulin}$ defined by the map ($\ref{niku}$) sends $\pi_{*}H^{2}(T,\mathbb{Z})$ to $\pi_{*}H^{2}(T,\mathbb{Z})$. We know that the ordinary transcendental lattices of an Abelian surface A and its Kummer surface X are Hodge isometric (up to a factor of 2) ~\cite{50, 60}
\begin{eqnarray}
\label{t}
	T(A)(2)\cong T(X).
\end{eqnarray}
The Hodge isometry ($\ref{t}$) can be enhanced by ($\ref{talpha}$) to a Hodge isometry $T(A, \alpha_{B_{A}})(2)\cong T(X, \alpha_{B})$.
\end{proof}

So we have natural isometries of the above transcendental lattices for B-fields associated with orbifold CFTs. Compare also ~\cite{70}.\\  

Let us first consider unwisted surfaces with B-field $B\in NS(X)\otimes\mathbb{R}$. We consider a algebraic K3 surface X following \cite{90} and use the Mukai pairing on the integral cohomology lattice. We denote the bounded derived categories of coherent sheaves on X by $D^{b}(X):=D^{b}(\text{Coh }X)$. Let $NS(X)$ be the N\'eron-Severi lattice. We introduce the lattice $\mathcal{N}(X)=H^{0}(X,\mathbb{Z})\oplus NS(X)\oplus H^{4}(X,\mathbb{Z})$. Recall that the Mukai vector $v(E)$ of an object $E\in D^{b}(X)$ is defined by
\begin{eqnarray}
v(E)=(r(E),c_{1}(E),s(E))=ch(E)\sqrt{td(X)}\in \mathcal{N}(X) \nonumber
\end{eqnarray}
where $ch(E)$ is the Chern character and $s(E)=ch_{2}(E)+r(E)$. We define an open subset
\begin{eqnarray}
\mathcal{P}(X)\subset \mathcal{N}(X)\otimes \mathbb{C} \nonumber
\end{eqnarray}  
consisting of vectors whose real and imaginary part span positive definite two-planes in $\mathcal{N}(X)\otimes \mathbb{R}$. $\mathcal{P}(X)$ consists of two connected components that are exchanged by complex conjugation. We have a free action of $GL^{+}(2,\mathbb{R})$ by the identification $\mathcal{N}(X)\otimes \mathbb{C}\cong\mathcal{N}(X)\otimes \mathbb{R}^{2}$. A section of this action is provided by the submanifold
\begin{eqnarray}
\label{q}
\mathcal{Q}(X)=\left\{\mho\in\mathcal{P}(X)|\left\langle \mho,\mho\right\rangle=0, \left\langle \mho,\bar{\mho}\right\rangle>0,r(\mho)=1\right\}\subset \mathcal{N}(X)\otimes \mathbb{C}. \nonumber
\end{eqnarray}   
$r(\mho)$ projects $\mho\in\mathcal{N}(X)\otimes \mathbb{C}$ into $H^{0}(X,\mathbb{C})$. We can identify $\mathcal{Q}(X)$ with the tube domain
\begin{eqnarray}
\left\{B+i\omega\in NS(X)\otimes\mathbb{C}|\omega^{2}>0\right\} \nonumber
\end{eqnarray}     
by
\begin{eqnarray}
\mho=exp(B+i\omega)=\upsilon^{0}+B+i\omega+\frac{1}{2}(B^{2}-\omega^{2})\upsilon+i\left\langle B,\omega\right\rangle\upsilon \nonumber
\end{eqnarray}
with $\upsilon^{0}=1\in H^{0}(X,\mathbb{Z})$ and $\upsilon=\left[pt\right]\in H^{4}(X,\mathbb{Z})$. We denote $\mathcal{P}^{+}(X)\subset \mathcal{P}(X)$ the connected component containing vectors of the form $exp(B+i\omega)$ for an ample $\mathbb{R}$-divisor class $\omega\in NS(X)\otimes \mathbb{R}$. Let $\Delta(X)=\left\{\delta\in\mathcal{N}(X)|\left\langle \delta,\delta\right\rangle=-2\right\}$ be the root system. For each $\delta\in\Delta(X)$ we have a complex hyperplane
\begin{eqnarray}
\delta^{\bot}=\left\{\mho\in\mathcal{N}(X)\otimes \mathbb{C}|\left\langle \mho,\delta\right\rangle=0\right\}\subset\mathcal{N}(X)\otimes \mathbb{C}. \nonumber
\end{eqnarray}
We denote by
\begin{eqnarray}
\mathcal{P}^{+}_{0}(X)=\mathcal{P}^{+}(X)\backslash\bigcup_{\delta\in\Delta(X)}\delta^{\bot}\subset\mathcal{N}(X)\otimes \mathbb{C}. \nonumber
\end{eqnarray}
 
Note that there are no spherical objects in $D^{b}(A)$ on an Abelian surface A \cite{80}. 

\begin{prop}
\label{proposition}
Let A be an Abelian surface and $X=\text{Km A}$ the corresponding Kummer surface. Then we have an embedding $\mathcal{P}^{+}(A)\hookrightarrow\mathcal{P}^{+}_{0}(X)$.   
\end{prop}            
\begin{proof}
An element of $\mathcal{P}^{+}(A)$ is of the form $exp(B+i\omega)\circ g$ for $g\in GL^{+}(2,\mathbb{R})$, $B\in NS(A)\otimes\mathbb{R}$ and $\omega\in NS(A)\otimes\mathbb{R}$ with $\omega^{2}>0$ \cite{80}. Let $\pi_{*}$ be the map induced by the rational map $\pi:A\rightarrow X$. The action of $GL^{+}(2,\mathbb{R})$ and the map $\pi_{*}$ commute. By Lemma $\ref{lattice}$ we have an injective map
\begin{eqnarray}
\label{pistar}
i: H^{even}(A,\mathbb{Z})\otimes\mathbb{R}\hookrightarrow H^{even}(X,\mathbb{Z})\otimes\mathbb{R}.
\end{eqnarray}
The 2-plane $\Omega$ given by the complex structure of the Abelian surface A defines the complex structure on $X$ by the 2-plane $\pi_{*}\Omega$. Therefore $\mathcal{N}(A)$ is mapped to $\mathcal{N}(X)$ and we get an induced map form $\mathcal{P}(A)$ to $\mathcal{P}(X)$. The proof of Theorem $\ref{nw}$ shows that vectors of the form $1/2\pi_{*}(exp(B_{T}+i\omega))$ for $B_{T},\omega\in NS(A)\otimes\mathbb{R}$  are sent to vectors
\begin{eqnarray}
\label{fuck}
\hat{\upsilon}^{0}+B+\frac{1}{2}\left(B^{2}-\left(\frac{1}{2}\pi_{*}\omega\right)^{2}\right)\hat{\upsilon}+i\left(\frac{1}{2}\pi_{*}\omega+\left\langle B, \frac{1}{2}\pi_{*}\omega\right\rangle\hat{\upsilon}\right) 
\end{eqnarray}
in $\mathcal{N}(X)\otimes\mathbb{C}$ with B as in Lemma $\ref{meinlemma}$. The elements of $\mathcal{N}(X)$ are contained in the orthogonal complement of $H^{2,0}(X)=\mathbb{C}[\pi_{*}\Omega]$ where $\pi_{*}\Omega=\pi_{*}\Omega_{1}+i\pi_{*}\Omega_{2}$.\footnote{By abuse of notation we denote the holomorphic two-form defining the complex structure and the 2-plane defined by it with the same symbol.} By corollary $\ref{korollar}$ we know that there are no roots of $H^{even}(X,\mathbb{Z})$ in the orthogonal complement of the 4-plane spanned by $\pi_{*}\Omega_{1}, \pi_{*}\Omega_{2}$ and the real and imaginary part of a vector of the form $(\ref{fuck})$ in $H^{even}(X,\mathbb{Z})\otimes\mathbb{R}$. Since $\pi_{*}\omega$ is an orbifold ample class in the closure of the ample cone, this proves the proposition.         
\end{proof}

The results of ~\cite{90} can be generalized for twisted surfaces ~\cite{100}. Any class $\alpha\in Br(X)=H^{2}(X,\mathcal{O}^{*}_{X})_{tor}$ can be represented by a $\check{C}ech$ 2-cocycle $\left\{\alpha_{ijk}\in \Gamma(U_{i}\cap U_{j}\cap U_{k}, \mathcal{O}_{X}^{*})\right\}$ on an analytic open cover $\lbrace U_{i}\rbrace$ of X.

\begin{defn}
An \textit{$(\alpha_{ijk})$-twisted coherent sheaf E} consists of pairs $(\left\{E_{i}\right\},\left\{\varphi_{ij}\right\})$ such that $E_{i}$ is a coherent sheaf on $U_{i}$ and $\varphi_{ij}:E_{j}|_{U_{i}\cap U_{j}}\rightarrow E_{i}|_{U_{i}\cap U_{j}}$ are isomorphisms satisfying the following conditions:

\begin{enumerate}
\item $\varphi_{ii}=id$\
\item $\varphi_{ji}=\varphi_{ij}^{-1}$\
\item $\varphi_{ij}\circ\varphi_{jk}\circ\varphi_{ki}=\alpha_{ijk}\cdot id$.
\end{enumerate}
\end{defn}
We denote the equivalence class of such Abelian categories of twisted coherent sheaves by $Coh(X,\alpha)$ and the bounded derived category by $D^{b}(X,\alpha)$. For details consult ~\cite{30}. For a realization of the following notions one has to fix a B-field lift B of the Brauer class $\alpha$ such that $\alpha=\alpha_{B}=exp(B^{0,2})$. The twisted Chern character
\begin{eqnarray}
	ch^{B}:D^{b}(X,\alpha_{B})\longrightarrow \widetilde{H}(X,B,\mathbb{Z}) \nonumber
\end{eqnarray}
introduced in ~\cite{160} identifies the numerical Grothendieck group with the twisted N\'eron-Severi group $NS(X,\alpha_{B}):=\widetilde{H}^{1,1}(X,B,\mathbb{Z})$. As in the untwisted case we denote by 
\begin{eqnarray}
\mathcal{P}(X,\alpha_{B})\subset NS(X,\alpha_{B})\otimes\mathbb{C} \nonumber
\end{eqnarray}
the open subset of vectors whose real and imaginary part span a positive plane in $NS(X,\alpha_{B})\otimes\mathbb{R}$. Let $\mathcal{P}^{+}(X,\alpha_{B})\subset \mathcal{P}(X,\alpha_{B})$ be the component containing vectors of the form $\exp(B+i\omega)$, where $B\in H^{2}(X,\mathbb{Q})$ is a B-field lift of $\alpha$ and $\omega$ a real ample class. $NS(A,\alpha_{B_{A}})$ is embedded into $NS(X,\alpha_{B})$, since we have
$\pi_{*}\Omega+\left\langle B, \pi_{*}\Omega\right\rangle\hat{\upsilon}=\pi_{*}(\Omega+\left\langle B_{A},\Omega\right\rangle\upsilon)$. Therefore Proposition 5.1 generalizes with similar arguments as above to 
\begin{prop}
Let $(A,\alpha_{B_{A}})$ be a twisted Abelian surface and $(X,\alpha_{B})$ the twisted Kummer surface with X the Kummer surface of A and B-field lifts as in Lemma $\ref{meinlemma}$. Then we have an embedding $\mathcal{P}^{+}(A,\alpha_{B_{A}})\hookrightarrow\mathcal{P}_{0}^{+}(X,\alpha_{B})$.   
\end{prop}

\subsection{Bridgeland stability conditions}

Bridgeland introduces stability conditions on a triangulated category $\mathcal{D}$ \cite{80}. For a review see \cite{85}. In our case this will be the bounded derived categories of coherent sheaves $D^{b}(X):=D^{b}(\text{Coh }X)$ on an Abelian or a K3 surface X. We denote by $K(\mathcal{D})$ the corresponding Grothendieck group of $\mathcal{D}$.

\begin{defn}~\cite{80}
\label{bridgeland}
A \textit{stability condition on a triangulated category} $\mathcal{D}$ consists of a group homomorphism $Z:K(\mathcal{D})\rightarrow \mathbb{C}$ called the \textit{central charge} and of full additive subcategories $\mathcal{P}(\phi)\subset\mathcal{D}$ for each $\phi\in \mathbb{R}$, satisfying the following axioms:
\begin{enumerate}
\item if $0\neq E\in\mathcal{P}(\phi)$, then $Z(E)=m(E)exp(i\pi\phi)$ for some $m(E)\in \mathbb{R}_{>0}$;
\item $\forall \phi\in\mathbb{R}, \mathcal{P}(\phi+1)=\mathcal{P}(\phi)\left[1\right]$;
\item if $\phi_{1}>\phi_{2}$ and $A_{j}\in\mathcal{P}(\phi_{j})$, then $Hom_{\mathcal{D}}(A_{1},A_{2})=0;$
\item for $0\neq E\in\mathcal{D}$, there is a finite sequence of real numbers $\phi_{1}>\cdots>\phi_{n}$ and a collection of triangles 
$$E_{i-1}\longrightarrow E_{i}\longrightarrow A_{i}$$ with $E_{0}=0$, $E_{n}=E$ and $A_{j}\in\mathcal{P}(\phi_{j})$ for all j. 
\end{enumerate}
\end{defn} 

A stability function on an Abelian category $\mathcal{A}$ is a group homomorphism $Z:K(\mathcal{A})\rightarrow\mathbb{C}$ such that for any nonzero $E\in\mathcal{A}$, $Z(E)$ lies in $H:=\left\{0\neq z\in\mathbb{C}| z/\left|z\right|=exp(i\pi\phi) \text{ with } 0<\phi\leq 1)\right\}$. 

\begin{defn}~\cite{38}
A \textit{t-structure on a triangulated category $\mathcal{D}$} is a pair of strictly full subcategories $(\mathcal{D}^{\leq 0},\mathcal{D}^{\geq 0})$ such that with $\mathcal{D}^{\leq n}=\mathcal{D}^{\leq 0}[-n]$ and $\mathcal{D}^{\geq n}=\mathcal{D}^{\geq 0}[-n]$:

\begin{enumerate}
\item  $\mathcal{D}^{\leq 0}\subset \mathcal{D}^{\leq 1}$ and $\mathcal{D}^{\geq 1}\subset \mathcal{D}^{\geq 0}$,
\item $Hom(X,Y)=0$ for $X\in \text{Ob }\mathcal{D}^{\leq 0}, Y\in \text{Ob }\mathcal{D}^{\geq 1}$,
\item For any $X\in \text{Ob }\mathcal{D}$ there is a distinguished triangle $A\rightarrow X\rightarrow B\rightarrow A[1]$ with $A\in \text{Ob }\mathcal{D}^{\leq 0}, B\in\text{Ob }\mathcal{D}^{\geq 1}$.
\end{enumerate}
\end{defn}

The heart of the t-structure is the full subcategory $\mathcal{D}^{\geq 0}\cap \mathcal{D}^{\leq 0}$. Important for the construction of stability conditions is

\begin{prop}~\cite{80}
To give a stability condition on a triangulated category $\mathcal{D}$ is equivalent to giving a bounded t-structure on $\mathcal{D}$ and a stability function on its heart which has the Harder-Narasimhan property.
\end{prop}

We recall some results of \cite{80}. The subcategory $\mathcal{P}(\phi)$ is Abelian and its nonzero objects are said to be semistable of phase $\phi$ for a stability condition $\sigma=(Z,\mathcal{P})$. We call its simple objects stable. The objects $A_{i}$ in Definition $\ref{bridgeland}$ are called semistable factors of E with respect to $\sigma$. We write $\phi^{+}_{\sigma}:=\phi_{1}$ and $\phi^{-}_{\sigma}:=\phi_{n}$. The mass of E is defined to be $m_{\sigma}(E)=\sum_{i}\left|Z(A_{i})\right|\in\mathbb{R}$. A stability condition is locally-finite if there exists some $\epsilon>0$ such that for all $\phi\in\mathbb{R}$ each quasi-Abelian subcategory $\mathcal{P}((\phi-\epsilon,\phi+\epsilon))$ is of finite length. In this case $\mathcal{P}(\phi)$ is of finite length and every semistable object has a finite Jordan-Holder filtration into stable objects of the same phase.\\
The set $Stab(\mathcal{D})$ of locally finite stability conditions on a triangulated category $\mathcal{D}$ has a topology induced by the generalised metric\footnote{This generalised metric has the usual properties of a metric but can take the value $\infty$.}:
\begin{eqnarray}
d(\sigma_{1},\sigma_{2})=sup_{0\neq E\in\mathcal{D}} \left\{\left|\phi^{-}_{\sigma_{2}}(E)-\phi^{-}_{\sigma_{1}}(E)\right|, \left|\phi^{+}_{\sigma_{2}}(E)-\phi^{+}_{\sigma_{1}}(E)\right|, \left|log \frac{m_{\sigma_{2}}(E)}{m_{\sigma_{1}}(E)}\right|\right\}. \nonumber
\end{eqnarray}
There is an action of the group of auto equivalences $Aut(\mathcal{D})$ of the derived category $\mathcal{D}$ on $Stab(\mathcal{D})$. For $\sigma=(Z,\mathcal{P})\in Stab(\mathcal{D})$ and $\Phi\in Aut(\mathcal{D})$ define the new stability condition $\Phi(\sigma)=(Z\circ\Phi_{*}^{-1},\mathcal{P}')$ with $\mathcal{P}'(\phi)=\Phi(\mathcal{P}(\phi))$. Here $\Phi_{*}$ is the induced automorphism of $K(\mathcal{D})$ of $\Phi$. Note that auto equivalences preserve the generalised metric.\\
The universal covering $\widetilde{GL^{+}(2,\mathbb{R})}$ of $GL^{+}(2,\mathbb{R})$ acts on the metric space $Stab(\mathcal{D})$ on the right in the following way: Let $\left(G,f\right)\in\widetilde{GL^{+}(2,\mathbb{R})}$ with $G\in GL^{+}(2,\mathbb{R})$ and an increasing function $f:\mathbb{R}\rightarrow \mathbb{R}$ with $f(\phi+1)=f(\phi)+1$ such that $G exp(i\pi\phi)/\left| exp(i\pi\phi)\right|=exp(2i\pi f(\phi))$ for all $\phi\in\mathbb{R}$. A pair $(G,f)\in \widetilde{GL^{+}(2,\mathbb{R})}$ maps $\sigma=(Z,\mathcal{P})\in Stab(\mathcal{D})$ to $(Z',P')=(G^{-1}\circ Z,\mathcal{P}\circ f)$.\\
The subgroup $\mathbb{C}\hookrightarrow \widetilde{GL^{+}(2,\mathbb{R})}$ acts freely on $Stab(\mathcal{D})$ for a triangulated category $\mathcal{D }$ by sending a complex number $\lambda$ and a stability condition $(Z,\mathcal{P})$ to a stability condition $(Z', \mathcal{P}')$ where $Z'(E)=exp(-i\pi\lambda)Z(E)$ and $\mathcal{P}'(\phi)=\mathcal{P}(\phi+ Re(\lambda))$. Note that this is for $\lambda=n\in\mathbb{Z}$ just the action of the shift functor $[n]$.

We are interested in the bounded derived category of coherent sheaves $D^{b}(X)$ on a smooth projective variety X over the complex numbers. In this case we say a stability condition is numerical if the central charge $Z:K(X)\rightarrow\mathbb{C}$ factors through the quotient group $\mathcal{N}(X)=K(X)/K(X)^{\bot}$.  Let us write $Stab(X)$ for the set of all locally finite numerical stability conditions on $\mathcal{D}^{b}(X)$. The Euler form $\chi$ is non-degenerate on $\mathcal{N}(X)\otimes\mathbb{C}$, so the central charge takes the form
\begin{eqnarray}
Z(E)=-\chi(p(\sigma),v(E)) \nonumber
\end{eqnarray}          
for some vector $p(\sigma)\in\mathcal{N}(X)\otimes\mathbb{C}$, defining a map $p:Stab(X)\longrightarrow \mathcal{N}(X)\otimes \mathbb{C}$. We have the following important theorem
\begin{thm}\cite{80} 
For each connected component $Stab^{*}(X)\subset Stab(X)$, there is a linear subspace  $V\subset \mathcal{N}(X)\otimes \mathbb{C}$ such that
\begin{eqnarray}
p:Stab^{*}(X)\longrightarrow \mathcal{N}(X)\otimes \mathbb{C} \nonumber
\end{eqnarray}  
is a local homeomorphism onto an open subset of the subspace V. In particular, $Stab^{*}(X)$ is a finite-dimensional complex manifold.
\end{thm}

We have the following description of the stability manifold for algebraic K3 surfaces:
\begin{thm}\cite{90}
\label{bridgeland2}
There is a distinguished connected component $Stab^{\dagger}(X)\subset Stab(X)$ which is mapped by $p$ onto the open subset $\mathcal{P}_{0}^{+}(X)$. The induced map $p: Stab^{\dagger}(X)\rightarrow \mathcal{P}_{0}^{+}(X)$ is a covering map. We denote by $Aut^{\dagger}_{0}(D^{b}(X))$ the subgroup of cohomological trivial auto equivalences of $D^{b}(X)$ which preserve the connected component $Stab^{\dagger}(X)$.  $Aut^{\dagger}_{0}(D^{b}(X))$ acts freely on $Stab^{\dagger}(X)$ and is the group of deck transformations of this covering. 
\end{thm}

The main difference in the case of Abelian surfaces is the absence of spherical objects. In fact there are no ill-behaved SCFTs on complex tori. For an Abelian surface A the Todd class is trivial thus the Mukai vector of an object $E\in D^{b}(A)$ is
\begin{eqnarray}
v(E)=(r(E),c_{1}(E),ch_{2}(E))\in\mathcal{N}(A)=H^{0}(A,\mathbb{Z})\oplus NS(A)\oplus H^{4}(A,\mathbb{Z}). \nonumber
\end{eqnarray} 
We define $\mathcal{P}^{+}(A)\subset \mathcal{N}(A)\otimes \mathbb{C}$ to be the component of the set of vectors which span positive-definite two-planes containing vectors of the form $exp(B+i\omega)$ with $B,\omega\in NS(A)\otimes\mathbb{R}$ and $\omega$ ample.
\begin{thm}~\cite{90}
Let A be an Abelian surface. Then there is a connected component $Stab^{\dagger}(A)\subset Stab(A)$ which is mapped by $p$ onto the open subset $\mathcal{P}^{+}(A)\subset\mathcal{N}(X)\otimes \mathbb{C}$, the induced map
\begin{eqnarray}
	p: Stab^{\dagger}(A)\longrightarrow \mathcal{P}^{+}(A)
\end{eqnarray}
is the universal cover, and the group of deck transformations is generated by the double shift-functor.
\end{thm} 

The fundamental group $\pi_{1}(\mathcal{P}^{+}(A))\cong \mathbb{Z}$ is generated by the loop induced by the $\mathbb{C}^{*}$ action on $\mathcal{P}(A)$.\\

We give an example of a stability condition on an algebraic K3 or an Abelian surface. For this we have to introduce a little more machinery. The standard t-structure of the derived category of coherent sheaves of a smooth projective variety has as its heart the Abelian category of coherent sheaves. For a K3 surface slope stability with this t-structure defines no stability condition since the stability function for any sheaf supported in dimension zero vanishes. The next simplest choice is the t-structure obtained by tilting ~\cite{130}. For details see ~\cite{80}.
\begin{defn}
A \textit{torsion pair} in an Abelian category $\mathcal{A}$ is a pair of full subcategories $(\mathcal{T}, \mathcal{F})$ satisfying
\begin{enumerate}
\item $Hom_{\mathcal{A}}(T,F)=0$ for all $T\in\mathcal{T}$ and $F\in\mathcal{F}$;
\item every object $E\in\mathcal{A}$ fits into a short exact sequence
\begin{eqnarray}
0\longrightarrow T\longrightarrow E\longrightarrow F\longrightarrow 0 \nonumber
\end{eqnarray}
for some pair of objects $T\in\mathcal{T}$ and $F\in\mathcal{F}$.
\end{enumerate} 
\end{defn}  
Then we have the following
\begin{prop} ~\cite{130}
Let $\mathcal{A}$ be the heart of a bounded t-structure on a triangulated category $\mathcal{D}$. Denote by $H^{i}(E)\in\mathcal{A}$ the i-th cohomology object of E with respect to this t-structure. Let $(\mathcal{T}, \mathcal{F})$ be a torsion pair in $\mathcal{A}$. Then the full subcategory
\begin{eqnarray}
\mathcal{A}^{*}=\left\{E\in \mathcal{D}| H^{i}(E)=0 \text{ for } i\notin \lbrace-1,0\rbrace, H^{-1}(E)\in\mathcal{F}, H^{0}(E)\in \mathcal{T}\right\} \nonumber
\end{eqnarray}
is the heart of a bounded t-structure on $\mathcal{D}$. 
\end{prop}
We say $\mathcal{A}^{*}$ is obtained from $\mathcal{A}$ by \textit{tilting} with respect to the torsion pair $(\mathcal{T}, \mathcal{F})$.\\

Let $\omega\in NS(X)\otimes\mathbb{R}$ be an element of the ample cone Amp(X) of an Abelian or an algebraic K3 surface X. We define the slope $\mu_{\omega}(E)$ of a torsion-free sheaf E on X to be
\begin{eqnarray}
	\mu_{\omega}(E)=\frac{c_{1}(E)\cdot\omega}{r(E)}. \nonumber
\end{eqnarray}

Let $\mathcal{T}$ be the category consisting of sheaves whose torsion-free part have $\mu_{\omega}$-semistable Harder-Narasimhan factors with $\mu_{\omega}>B\cdot\omega$ and $\mathcal{F}$ the category consisting of torsion-free sheaves with $\mu_{\omega}$-semistable Harder-Narasimhan factors with $\mu_{\omega}\leq B\cdot\omega$. $(\mathcal{T},\mathcal{F})$ defines a torsion pair. Tilting with respect to this torsion pair gives a bounded t-structure on $D^{b}(X)$ with heart $\mathcal{A}(B,\omega)$ that depends on $B\cdot\omega$. As stability function on this heart we choose
\begin{eqnarray}
\label{string}
	Z_{(B,\omega)}(E)=(exp(B+i\omega),v(E)).
\end{eqnarray}
Note that the central charge ($\ref{string}$) is of the form guessed by physicists  by mirror symmetry arguments. For a Calabi-Yau threefold we expect quantum corrections for this central charge ~\cite{150}.   
\begin{prop}~\cite{90}
\label{example}
The pair $(Z_{(B,\omega)},\mathcal{A}(B,\omega))$ defines a stability condition if for all spherical sheaves E on X one has $Z(E)\notin\mathbb{R}_{\leq0}$. In particular, this holds whenever $\omega^{2}>2$.
\end{prop} 

We denote the set of all stability conditions arising in this way by $V(X)$. We denote by $\Delta^{+}(X)\subset \Delta(X)$ elements $\delta\in\Delta(X)$ with $r(\delta)>0$.  We define the following subset of $\mathcal{Q}(X)$
\begin{eqnarray}
\mathcal{L}(X)=\left\{\Omega=exp(B+i\omega)\in\mathcal{Q}(X)|\omega\in Amp(X), \left\langle \Omega,\delta\right\rangle\notin \mathbb{R}_{\leq 0}, \forall \delta\in \Delta^{+}(X)\right\}.  \nonumber
\end{eqnarray} 
The map $p$ restricts to a homeomorphism ~\cite{90} 
\begin{eqnarray}
\label{homeo}
p:V(X)\longrightarrow\mathcal{L}(X). \nonumber
\end{eqnarray} 

We use the free action of $\widetilde{GL^{+}(2,\mathbb{R})}$ on $V(X)$ to introduce $U(X):=V(X)\cdot\widetilde{GL^{+}(2,\mathbb{R})}$. The connected component $Stab^{\dagger}(X)$ is the unique one containing $U(X)$. $U(X)$ can be described as the stability conditions in $Stab^{\dagger}(X)$ for which all skycraper sheaves $\mathcal{O}_{p}$ are stable of the same phase ~\cite{80}. Since we have no spherical objects on an Abelian surface A in this case we have $Stab^{\dagger}(A)=U(A)$. \\

We say a set of objects $S\subset D^{b}(X)$ has bounded mass in a connected component $Stab^{*}(X)\subset Stab(X)$ if $sup\left\{m_{\sigma}(E)|E\in S\right\}<\infty$ for some point $\sigma\in Stab^{*}(X)$. This implies that the set of Mukai vectors  $\left\{v(E)|E\in S\right\}$ is finite. We have a wall-and-chamber structure:
\begin{prop}\cite{90}
Suppose that the subset $S\subset D^{b}(X)$ has bounded mass in $Stab^{*}(X)$ and fix a compact subset $B\subset Stab^{*}(X)$. Then there is a finite collection $\left\{W_{\gamma}|\gamma\in\Gamma\right\}$ of real codimension-one submanifolds of $Stab^{*}(X)$ such that any component
\begin{eqnarray}
C\subset B\backslash\bigcup_{\gamma\in\Gamma}W_{\gamma} \nonumber
\end{eqnarray}
has the following property: if $E\in S$ is $\sigma-$semistable for $\sigma\in C$, then E is $\sigma$-semistable for all $\sigma\in C$. Moreover, if $E\in S$ has primitive Mukai vector, then E is $\sigma$-stable for all $\sigma\in C$. 
\end{prop}     
Using this result Bridgeland proved the following theorem for the boundary $\partial U(X)$ of the open subset $U(X)$ that is contained in a locally finite union of codimension-one real submanifolds of Stab(X):
\begin{thm}\cite{90}
\label{ux}
Suppose that $\sigma\in \partial U(X)$ is a general point of the boundary of U(X), i.e. it lies on only one codimension-one submanifold of $Stab(X)$. Then exactly one of the following possibilities holds:
\begin{enumerate}
\item There is a rank r spherical vector bundle A such that the only $\sigma$-stable factors of the objects $\left\{\mathcal{O}_{p}|p\in X\right\}$ are A und $T_{A}(\mathcal{O}_{p})$. Thus the Jordan-Holder filtration of each $\mathcal{O}_{p}$ is given by
\begin{eqnarray}
0\longrightarrow A^{\oplus r}\longrightarrow\mathcal{O}_{p}\longrightarrow T_{A}(\mathcal{O}_{p})\longrightarrow 0. \nonumber
\end{eqnarray}
\item There is a rank r spherical vector bundle A such that the only $\sigma$-stable factors of the objects $\left\{\mathcal{O}_{p}|p\in X\right\}$ are $A\left[2\right]$ and $T_{A}^{-1}(\mathcal{O}_{p})$. Thus the Jordan-Holder filtration of each $\mathcal{O}_{p}$ is given by
\begin{eqnarray}
0\longrightarrow T_{A}^{-1}(\mathcal{O}_{p})\longrightarrow\mathcal{O}_{p}\longrightarrow A^{\oplus r}\left[2\right]\longrightarrow 0. \nonumber
\end{eqnarray}
\item There are a nonsingular rational curve $C\subset X$ and an integer k such that $\mathcal{O}_{p}$ is $\sigma$-stable for $p\notin C$ and such that the Jordan-Holder filtration of $\mathcal{O}_{p}$ for $p\in C$ is
 \begin{eqnarray}
0\longrightarrow \mathcal{O}_{C}(k+1)\longrightarrow\mathcal{O}_{p}\longrightarrow \mathcal{O}_{C}(k)\left[1\right]\longrightarrow 0. \nonumber
\end{eqnarray}
\end{enumerate} 
\end{thm}          
Here $T_{A}(B)$ is the Seidel-Thomas twist of B with respect to the spherical object A \cite{170}.

\subsection{Inducing stability conditions}

Let A be an Abelian surface and $X=\text{Km A}$ the associated Kummer surface. Then Proposition \ref{proposition} and Theorem \ref{bridgeland2} imply that for every $z\in i(\mathcal{P}^{+}(A))$ there is a stability condition $\sigma\in Stab^{\dagger}(X)$ with $p(\sigma)=z$. Here $i$ is the injective linear map defined in the proof of Proposition $\ref{proposition}$ and we consider the map $p:Stab^{*}(X)\longrightarrow \mathcal{N}(X)\otimes \mathbb{C}$. We observed in Theorem \ref{nw} that a four-plane defining a SCFT on a two-dimensional complex torus T with B-field $B_{T}$ and K\"ahler class $\omega$ is mapped to a four-plane defining a SCFT with B-field $B=\frac{1}{2}\pi_{*}B_{T}+\frac{1}{2}B_{\mathbb{Z}}$. $\pi_{*}\omega$ is an orbifold ample class orthogonal to the 16 classes $\left\{\hat{E}_{i}\right\}, i\in\mathbb{F}_{2}^{4}$. $\pi_{*}\omega$ is an element of the closure of the ample cone $\overline{Amp(X)}=Nef(X)$. We assume $\omega^{2}>1$. By the covering map property there is a stability condition $\sigma$ with $\pi(\sigma)=exp(B+i\pi_{*}\omega)$ on the boundary of $U(X)$. Since this stability condition lies on the boundary of $U(X)$ there must be some points $p\in X$ such that $\mathcal{O}_{p}$ is unstable with respect to $\sigma$. Every (-2) curve defines a boundary element of $U(X)$ as in the third case of Theorem $\ref{ux}$ \cite{190}. This gives
          
\begin{lem}
Let $exp(B+i\pi_{*}\omega)\in i(\mathcal{P}^{+}(A))$ be as in Proposition $\ref{proposition}$ with $\omega^{2}>1$. Then there is a stability condition $\sigma\in \partial U(X)$ with $\pi(\sigma)=exp(B+i\pi_{*}\omega)$. This $\sigma$ is an element of the codimension-one submanifolds associated to the 16 exceptional divisor classes.   
\end{lem}   

The covering $p:Stab^{\dagger}(X)\rightarrow \mathcal{P}^{+}_{0}(X))$ is normal ~\cite{90}. 

\begin{prop}
There is an injective map from the group of deck transformations of $Stab^{\dagger}(A)$ to the group of deck transformations of $Stab^{\dagger}(X)$.
\end{prop}
\begin{proof}
The fundamental group $\pi_{1}(\mathcal{P}^{+}(A))\cong \pi_{1}(GL^{+}(2,\mathbb{R}))=\mathbb{Z}$ is a free cyclic group generated by the loop coming from the $\mathbb{C}^{*}$ action on $\mathcal{P}^{+}(A)$. This is represented by a rotation matrix in $GL^{+}(2,\mathbb{R})$. We choose base points $l,l'$ and $\sigma\in Stab^{\dagger}(X)$ with $p(\sigma)=l'$. The induced map 
\begin{eqnarray}
\pi_{1}(\mathcal{P}^{+}(A),l) \longrightarrow \pi_{1}(\mathcal{P}^{+}_{0}(X),l') \nonumber
\end{eqnarray}
is injective since the map $\pi_{*}$ and the action of $GL^{+}(2,\mathbb{R})$ commute. The trivial element of $\pi_{1}( \mathcal{P}^{+}(A),l)$ is the only normal subgroup mapped to the normal subgroup $p_{*}(\pi_{1}(Stab^{\dagger}(X),\sigma)$ of $\pi_{1}(\mathcal{P}^{+}_{0}(X),l')$. 
\end{proof}

From the discussion of the SCFT side of the story we expect that there is an embedding of the connected component $Stab^{\dagger}(A)$ into the
distinguished connected component $Stab^{\dagger}(X)$.

\begin{thm}
\label{meinprop1}
Let $Stab^{\dagger}(A)$ be the (unique) maximal connected component of the space of stability conditions of an Abelian surface A and $Stab^{\dagger}(X)$ the
distinguished connected component of Stab(X) of the Kummer surface X=Km A. Then every connected component of $p^{-1}(i(\mathcal{P}^{+}(A)))$ is homeomorphic to $Stab^{\dagger}(A)$.\\   
\end{thm}
\begin{proof}
Since we have a homeomorphism $i(\mathcal{P}^{+}(A))\cong \mathcal{P}^{+}(A)$ the fundamental group $\pi_{1}(i(\mathcal{P}^{+}(A)))=\mathbb{Z}$ is also a free cyclic group. Note that $\mathcal{P}^{+}(A)$ is path connected and locally path connected. We consider a path component of the covering space $p^{-1}(i(\mathcal{P}^{+}(A)))$ which is again a covering space. Since the generator of $\pi_{1}(i(\mathcal{P}^{+}(A)))$ lifts to the double shift functor [2] a path connected component of this covering space is simply connected and is thus isomorphic to $Stab^{\dagger}(A)$.
\end{proof}

Note that deck transformations except double shifts exchange the components of $p^{-1}(i(\mathcal{P}^{+}(A)))$. Theorem $\ref{meinprop1}$ defines embeddings $Stab^{\dagger}(A)\hookrightarrow Stab^{\dagger}(X)$. In fact, we get one embedding up to deck transformations by the uniqueness of lifts. We construct this embedding topologically. A functor embedding $Stab^{\dagger}(A)$ into $Stab^{\dagger}(X)$ was described in ~\cite{220}.

\begin{rem}
For a twisted Abelian surface $(A,\alpha_{B_{A}})$ and the twisted Kummer surface $(\mbox{Km A},\alpha_{B})$ with B-field lifts as in Lemma $\ref{meinlemma}$ a similar statement to Theorem $\ref{meinprop1}$ holds true. 
\end{rem}

\section*{Acknowledgements}
It is a pleasure to thank my adviser Katrin Wendland for generous support. I thank Heinrich Hartmann, Daniel Huybrechts and Emanuele Macr\`i for helpful discussions and correspondences. I am grateful to the organizers of the programme on moduli spaces 2011 at the Isaac Newton Institute in Cambridge. In particular, I thank Professor Richard Thomas from Imperial College in London for support. This research was partially supported by the ERC Starting Independent Researcher Grant StG No. 204757-TQFT (Katrin Wendland, PI).


\begin{thebibliography}{99}
\bibitem{5} K. Narain: New heterotic string theories in uncompactified dimensions $<10$, Phys. Lett. 169B, 41-46 (1986)
\bibitem{6} W. Nahm, K. Wendland: A Hikers Guide to K3: Aspects of N=(4,4) Superconformal Field Theory with central charge c=6, Comm.Math.Phys. 216, 85-138 (2001)
\bibitem{18} M. Khalid, K. Wendland: SCFTs on higher dimensional cousins of K3s, in preparation.
\bibitem{3} A. Bayer, E. Macr\`i, Y. Toda: Bridgeland stability conditions on threefolds I: Bogomolov-Gieseker type inequalities, arXiv:1103.5010.
\bibitem{7} K. Wendland: Moduli spaces of unitary conformal field theories, PhD thesis, university of Bonn, 2000.
\bibitem{8} K. Wendland: Consistency of Orbifold Conformal Field Theories on K3, Adv.Theor.Math.Phys. 5, 429-456 (2002) 
\bibitem{71} A. Kapustin, D. Orlov: Vertex algebras, mirror symmetry, and D-branes: the case of complex tori, Comm. Math. Phys. 233,79-136 (2003)
\bibitem{666} D. Huybrechts: Moduli spaces of HyperK\"ahler manifolds and mirror symmetry, in: Intersection Theory and Moduli. Proc. Trieste 2002, math.AG/0210219.
\bibitem{18001} I. Piateckii-Shapiro, I. Shafarevich: A Torelli theorem for algebraic structures of type K3, Math. USSR Izvetija 5, 547-587 (1971)
\bibitem{28} D. Burns Jr., M. Rapoport: On the Torelli theorem for K\"ahlerian K3 surfaces, Ann. Sci. \'Ecole Norm. Sup. (4) 8, 235-274 (1975)
\bibitem{9} K. Wendland: On the geometry of singularities in quantum field theory, in: Proceedings of the International Congress of Mathematicians, Hyderabad, August 19-27, 2010, Hindustan Book Agency, 2144-2170 (2010)
\bibitem{92} P. Aspinwall, D. Morrison: String theory on K3 surfaces, in: Mirror Symmetry II, AMS, Providence, RI 1997. 
\bibitem{923} E. Witten: String dynamics in various dimensions, Nucl.Phys. B443, 85-126 (1995)
\bibitem{930} D. Morrison: Geometry of K3 surfaces, lecture notes from 1988, http://www.cgtp.duke.edu/ITP99/morrison/cortona.pdf. 
\bibitem{935} A. Taormina, K. Wendland: The overarching finite symmetry group of Kummer surfaces in the Mathieu group $M_{24}$, arXiv:1107.3834v3. 
\bibitem{940} V. Nikulin: Finite automorphism groups of K\"ahler K3 surfaces, Trans. Mosc. Math. Soc. 38,71-135 (1980)
\bibitem{941} V. Nikulin: Integral symmetric bilinear forms and some of their applications, Math. USSR Isv. 14, 103-167 (1980)
\bibitem{1828} W. Barth, K. Hulek, C. Peters, A. Van De Ven: Compact Complex Surfaces, 2nd edition, Springer 2004.      
\bibitem{1830} D. Joyce: Compact Manifolds with Special Holonomy, Oxford University Press 2000.
\bibitem{1840} D. Morrison: Some remarks on the moduli of K3 surfaces, in: Classifications of Algebraic and Analytic Manifolds, Progress in Math. 39, Birkhauser, 303-332 (1983)
\bibitem{180} R. Kobayashi, A.N. Todorov: Polarized period map for generalized K3 surfaces and the moduli of Einstein metrics, Tohoku Math. J., 39, 341-363 (1987)
\bibitem{10} N. Hitchin: Generalized Calabi-Yau manifolds. Q. J. Math. 54, 281-308 (2003)
\bibitem{20} D. Huybrechts: Generalized Calabi-Yau structures, K3 surfaces, and B-fields. Int. J. Math. 16, 13-36 (2005)
\bibitem{30} A. Caldararu: Derived categories of twisted sheaves on Calabi-Yau manifolds, PhD thesis, Cornell 2000. 
\bibitem{33} J. Milne: \'Etale cohomology, Princeton 1980.
\bibitem{34} D. Huybrechts, S. Schroer: The Brauer group of analytic K3 surfaces, IMRN 50, 2687-2698 (2003)
\bibitem{40} D. Huybrechts: The Global Torelli Theorem: classical, derived, twisted, in: Algebraic geometry, Seattle 2005, Proc. of Symposia in Pure Mathematics, AMS, 235-258 (2009)
\bibitem{50} D. Morrison: On K3 surfaces with large Picard number, Invent. Math. 75, 105-121 (1984)
\bibitem{60} V. Nikulin: On Kummer surfaces, Math. USSR Izvestija 9, 261-275 (1975)
\bibitem{70} P. Stellari: Derived categories and Kummer surfaces, Math. Z. 256, 425-441 (2007)
\bibitem{80} T. Bridgeland: Stability conditions on triangulated categories, Ann. Math. 166, 317-346 (2007)
\bibitem{38} A. Beilinson, J. Bernstein, P. Deligne, Faisceaux Pervers, Ast\'erique 100, Soc. Math de France (1983)
\bibitem{85} E. Macr\`i: Stability conditions for derived categories, Appendix D of C. Bartocci, U. Bruzzo, and D. Hern\'andez-Ruip\'erez: Fourier-Mukai and Nahm transformations in geometry and mathematical physics, Birkh\"auser 2009.
\bibitem{90} T. Bridgeland: Stability conditions on K3 surfaces, Duke Math. J. 141, 241-291 (2008)
\bibitem{97} E. Macr\`i: Stability conditions on curves, Math. Res. Lett. 14, 657-672 (2007)
\bibitem{100} D. Huybrechts, E. Macr\`i, P. Stellari: Stability conditions for generic K3 surfaces, Compositio Math. 144, 134-162 (2008)
\bibitem{110} T. Bridgeland: Spaces of stability conditions, in: D. Abramovich et al. (eds.): Algebraic Geometry: Seattle 2005, Proc. of Symposia in Pure Mathematics, AMS, 1-22 (2009)
\bibitem{120} E. Macr\`i: Some examples of stability manifolds, math.AG/0411613.
\bibitem{130} D. Happel, I. Reiten, S.O. Smalo: Tilting in abelian categories and quasitilted algebras, Mem. Amer. Math. Soc. 120, no. 575 (1996)
\bibitem{150} P. Aspinwall: D-Branes on Calabi-Yau Manifolds, in: J. Maldacena (ed.): Progress in String Theory, World Scientific 2005.
\bibitem{160} D. Huybrechts, P. Stellari: Equivalences of twisted K3 surfaces, Math. Ann. 332, 901-936 (2005) 
\bibitem{170} P. Seidel, R. Thomas: Braid group actions on derived categories of coherent sheaves, Duke Math. J. 108, no. 1, 37-108 (2001)
\bibitem{190} H. Hartmann: Cusps of the K\"ahler moduli space and stability conditions on K3 surfaces, arXiv:1012.3121v1, to appear in Math. Ann.
\bibitem{200} J. Bernstein, V. Lunts: Equivariant sheaves and functors, Springer 1994.
\bibitem{220} E. Macr\`i, S. Mehrotra, P. Stellari: Inducing stability conditions, J. Alg. Geom. 18, 605-649 (2009)
\end{thebibliography}
\end{document}